\documentclass{amsart}
\usepackage{epsfig,amsmath,amsthm,amssymb,graphicx,enumerate,mathrsfs}
\usepackage[arrow,matrix,curve,cmtip,ps]{xy}
\usepackage{pb-diagram} 
\usepackage{pb-xy}
\usepackage{tikz} 
\usepackage{caption}
\usepackage{subcaption}
\usetikzlibrary{matrix,arrows} 

\usepackage{color}

\newtheorem{proposition}{Proposition}[section]
\newtheorem{theorem}[proposition]{Theorem}
\newtheorem*{theorem_main}{Main theorem}
\newtheorem{corollary}[proposition]{Corollary}
\newtheorem{lemma}[proposition]{Lemma}
\newtheorem*{proposition*}{Proposition}
\newtheorem{conjecture}[proposition]{Conjecture}
\theoremstyle{definition}
\newtheorem{definition}[proposition]{Definition}

\newtheorem{remark}[proposition]{Remark}

\newcommand{\br}{\text{b}}
\newcommand{\C}{\mathcal{C}}
\newcommand{\str}{{\operatorname{str}}}
\renewcommand{\top}{{\operatorname{top}}} 
\newcommand{\ex}{{\operatorname{ex}}}
\newcommand{\Ctop}{\C_\top}
\newcommand{\Cex}{\C_\ex}

\newcommand{\Chat}{\widehat{\C}}
\newcommand{\Sstr}{\S_{\str}}
\newcommand{\Shat}{\widehat{\S}}

\newcommand{\Z}{\mathbb{Z}}
\newcommand{\Q}{\mathbb{Q}}

\renewcommand{\S}{\mathcal{S}}

\newcommand{\K}{\mathcal{K}}

\newcommand{\bdry}{\ensuremath{\partial}}
\newcommand{\im}{\operatorname{Im}}
\newcommand{\id}{\operatorname{id}}
\newcommand{\maps}{\operatorname{Maps}}

\newcommand{\into}{\hookrightarrow}
\renewcommand{\H}{\mathcal{H}}
\newcommand{\bigslant}[2]{{\raisebox{.2em}{$#1$}\left/\raisebox{-.2em}{$#2$}\right.}}

\numberwithin{equation}{section}

\begin{document}
\title[Satellite operators as group actions]{Satellite operators as group actions\\on knot concordance}

\author{Christopher W.\ Davis}
\address{Department of Mathematics, The University of Wisconsin at Eau Claire}
\email{daviscw@uwec.edu}
\urladdr{http://people.uwec.edu/daviscw/}

\author{Arunima Ray$^{\dag}$}
\address{Department of Mathematics, Brandeis University}
\email{aruray@brandeis.edu}
\urladdr{http://people.brandeis.edu/$\sim$aruray/}

\thanks{$^{\dag}$Partially supported by NSF--DMS--1309081 and the Nettie S.\ Autrey Fellowship (Rice University)}

\date{\today}

\begin{abstract}Any knot in a solid torus, called a pattern or satellite operator, acts on knots in $S^3$ via the satellite construction. We introduce a generalization of satellite operators which form a group (unlike traditional satellite operators), modulo a generalization of concordance. This group has an action on the set of knots in homology spheres, using which we recover the recent result of Cochran and the authors that satellite operators with strong winding number~$\pm 1$ give injective functions on topological concordance classes of knots, as well as smooth concordance classes of knots modulo the smooth 4--dimensional Poincar\'{e} Conjecture. The notion of generalized satellite operators yields a characterization of surjective satellite operators, as well as a sufficient condition for a satellite operator to have an inverse. As a consequence, we are able to construct infinitely many non-trivial satellite operators $P$ such that there is a satellite operator $\overline{P}$ for which $\overline{P}(P(K))$ is concordant to $K$ (topologically as well as smoothly in a potentially exotic $S^3\times [0,1]$) for all knots $K$; we show that these satellite operators are distinct from all connected-sum operators, even up to concordance, and that they induce bijective functions on topological concordance classes of knots, as well as smooth concordance classes of knots modulo the smooth 4--dimensional Poincar\'{e} Conjecture. \end{abstract}

\maketitle

\section{Introduction}

The satellite operation is a classical and well-studied family of functions on the set of knots in $S^3$. Briefly, the satellite operation involves a \textit{satellite operator}, or \textit{pattern}, $P$, i.e.\ a knot in a solid torus $V=S^1\times D^2$, and a knot $K$ in $S^3$, called the \textit{companion}; the \textit{satellite knot} $P(K)$ is the image of the knot $P$ when the solid torus $V$ is tied into the knot $K$ (see Figure~\ref{fig:satelliteoperation} and Section~\ref{sec:background}).

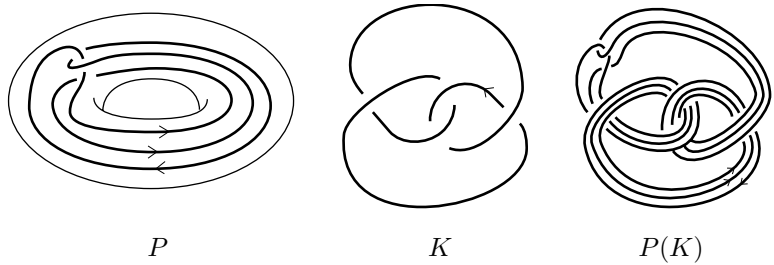
\begin{figure}[h!]
\begin{center}
\begin{picture}(4,1.15)
\put(0,0.25){\includegraphics{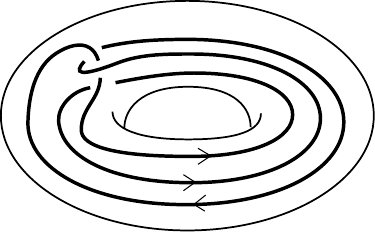}}
\put(0.73,-0.1){$P$}
\put(1.75,0.15){\includegraphics{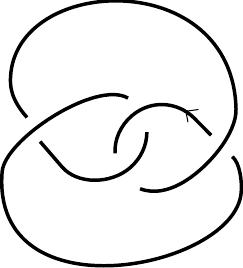}}
\put(2.2,-0.1){$K$}
\put(2.97,0.15){\includegraphics{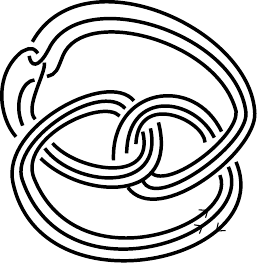}}
\put(3.3,-0.1){$P(K)$}
\end{picture}
\end{center}
\caption{The satellite operation on knots in $S^3$.}\label{fig:satelliteoperation}
\end{figure}

We will consider four different equivalence relations on knots, with corresponding sets of equivalence classes $\C$, $\Ctop$, $\Cex$, and $\C_{R}$, where $R$ is any localization of $\Z$. Here $\C$ denotes the \textit{smooth knot concordance group}, consisting of knots up to smooth concordance, where recall that two knots $K_0\hookrightarrow S^3\times\{0\}$ and $K_1\hookrightarrow S^3\times\{1\}$ are \textit{smoothly concordant} if they cobound a smooth, properly embedded annulus in $S^3\times [0,1]$ with its usual smooth structure. Similarly, $K_0\hookrightarrow S^3\times\{0\}$ and $K_1\hookrightarrow S^3\times\{1\}$ are \textit{topologically concordant} if they cobound a locally flat, properly embedded annulus in a topological manifold homeomorphic to $S^3\times [0,1]$; knots modulo topological concordance form the \textit{topological concordance group}, denoted $\Ctop$. As a transition of sorts between $\C$ and $\Ctop$, we have the \textit{exotic concordance group} $\Cex$, consisting of knots up to exotic concordance, where $K_0\hookrightarrow S^3\times\{0\}$ and $K_1\hookrightarrow S^3\times\{1\}$ are \textit{exotically concordant} if they cobound a smooth, properly embedded annulus in a smooth manifold homeomorphic to $S^3\times [0,1]$, but not necessarily diffeomorphic, i.e.\  $\Cex$ consists of knots up to concordance in a potentially exotic $S^3\times [0,1]$. For $R$ a localization of $\Z$, the knots $K_0\hookrightarrow S^3\times\{0\}$ and $K_1\hookrightarrow S^3\times\{1\}$ are $R$--concordant if they cobound a smooth, properly embedded annulus in a smooth $R$--homology cobordism from $S^3\times\{0\}$ to $S^3\times\{1\}$; $\C_R$ denotes the group of knots up to $R$--concordance. If the smooth 4--dimensional Poincar\'{e} Conjecture is true then $\Cex = \C$~\cite[Proposition 3.2]{CDR14}. In fact, for odd $n$, it is unknown whether $\C_{\Z[1/n]}=\C$. ($\C_{\Z[1/2]}$ is distinct from $\C$, since the figure eight knot is slice in a $\Z[\frac{1}{2}]$--homology ball, but not slice.) 

The satellite operation is well-defined on concordance classes of each of the types mentioned above, that is, any satellite operator $P$ gives a function, which we also refer to as a satellite operator, 
\begin{align*}
P:\C_*&\to\C_*\\
[K]&\mapsto[P(K)]
\end{align*}
for $*\in\{\varnothing, \ex,\top,R\}$ for $R$ any localization of $\Z$, where $[\,\cdot\,]$ denotes the relevant concordance class. This fact has been used to construct knots which yield distinct concordance classes but which cannot be distinguished between by many classical invariants. Examples of this philosophy in action can be found in \cite{CT04, CHL11}. In \cite{CFHeHo11}, winding number one satellite operators are used to construct non-concordant knots which have homology cobordant zero surgery manifolds. In the more general context of 3-- and 4--manifold topology, satellite operations were used in~\cite{H08} to modify a 3--manifold while fixing its homology type. Winding number one satellite operators, which are of particular interest in this paper, are related to Mazur 4--manifolds \cite{AkKir79} and Akbulut corks \cite{Ak91}. In \cite{AkYas08} it was shown that changing the attaching curve of a 2--handle in a handlebody description of a 4--manifold by a winding number one satellite operation can change the diffeomorphism type while fixing the homeomorphism type.

As a result, there has been considerable interest in understanding how satellite operators act on $\C$. For example, it is a famous conjecture that the Whitehead double of a knot $K$ is smoothly slice if and only if $K$ is smoothly slice~\cite[Problem 1.38]{kirbylist}. This question might be generalized to ask if satellite operators are injective on smooth concordance classes, that is, given a satellite operator $P$, does $P(K)=P(J)$ imply $K=J$ in smooth concordance? A survey of such work on the Whitehead doubling operator may be found in \cite{HeK12}. In \cite{CHL11}, several `robust doubling operators' were introduced and some evidence was provided for their injectivity. This is the current state of knowledge in the winding number zero case. For operators with nonzero winding numbers, there has been more success as seen in the following recent theorem. 

\begin{theorem}[Theorem~5.1 of \cite{CDR14}, see Corollary~\ref{cor:injectivecorollary}]\label{thm:cdrthm} Suppose $P$ is a pattern with non-zero winding number $n$. Then
\begin{itemize}
\item [(a)]  $P:\C_{\Z[\frac{1}{n}]}\to \C_{\Z[\frac{1}{n}]}$ is injective.
\end{itemize}
Suppose that $P$ is a pattern with strong winding number $\pm 1$. Then 
\begin{itemize}
\item [(b)]  $P:\C_{ex}\to \C_{ex}$ is injective,
\item [(c)]   $P:\C_{top}\to\C_{top}$ is injective, and
\item [(d)]  if the smooth 4-dimensional Poincar\'{e} Conjecture holds, $P:\C\to\C$ is injective.
\end{itemize}
\end{theorem}

See Section~\ref{sec:background} for a definition of strong winding number; for the moment it suffices to know that patterns with strong winding number $\pm 1$ are plentiful. The above theorem has a number of useful corollaries (see~\cite{CDR14}) and gives us a valuable tool in studying satellite knots. We will see in Section~\ref{sec:background} that it will follow easily from the main theorem of this paper. 

Let $\S$ denote the set of all satellite operators. $\S$ has a monoid structure with respect to which the usual satellite operation is a monoid action (see Section~\ref{sec:background}). The main technical result of this paper shows that the classical satellite operation is in fact a restriction of a natural group action. Specifically, we show that satellite operators form a submonoid of the group of homology cobordism classes of homology cylinders, introduced by Levine in~\cite{Lev01},  and we show that this group has a natural action on concordance classes of knots in homology spheres which is compatible with the classical satellite operation. In other words, we prove a theorem of the following type. 

\begin{theorem_main}[See Section~\ref{sec:background} for the precise statement]
Let $*\in\{\ex,\top,R\}$ for $R$ a localization of $\Z$. For each $\C_*$, there is a particular submonoid $\S_*$ of $\S$, an enlargement $$\Psi:\C_*\into \Chat_*,$$ and a monoid morphism, $$E:\S_*\to\Shat_*,$$ where $\Shat_*$ is a group which acts on $\,\Chat_*$, such that the following diagram commutes for all $P\in \S_*$:
$$
\begin{diagram}
\node{\C_*}\arrow{e,t}{P}\arrow{s,l,J}{\Psi}\node{\C_*}\arrow{s,l,J}{\Psi}
\\
\node{\Chat_*}\arrow{e,t}{E(P)}\node{\Chat_*}
\end{diagram}
$$
\end{theorem_main}

In the above, $\S_*$ contains all strong winding number $\pm 1$ operators for $* = \ex$ or $\top$, and all winding number $\pm n$ operators when $* = \Z[\frac{1}{n}]$;   
$\Chat_*$ consists of knots in homology spheres up to concordance in the category~$*$; and $\Shat_*$ is a group of homology cobordism classes of homology cylinders. For a pattern $P\subseteq V=S^1\times D^2$, $E(P)$ is just the exterior of a regular neighborhood of $P$ in $V$. See Section~\ref{sec:background} for the precise definitions, as well as the proof of the main theorem. 

Since $E(P)$ is an element of a group acting on $\Chat_*$, $E(P):\Chat_*\to\Chat_*$ must be a bijection. This observation yields Theorem~\ref{thm:cdrthm} as a corollary, via an elementary diagram chase (Corollary~\ref{cor:injectivecorollary}).

Moreover, considering the satellite operation as a restriction of a group action provides a novel approach to the problem of finding non-trivial surjective satellite operators on $\C_*$. While it is elementary to show that satellite operators with winding number other than $\pm1$ cannot give surjections on knot concordance (see Proposition~\ref{prop:n>1}), very little is known in the case of satellite operators of winding number $\pm1$. For instance, a conjecture of Akbulut \cite[Problem 1.45]{kirbylist} that there exists a winding number one satellite operator $P$ such that $P(K)$ is not slice for any knot $K$ (that is, the unknot is not in the image of the satellite operator) remains open. 

For $*\in \{\ex, \top, R\}$  and $P$ in the monoid $\S_*$, we have shown that $E(P)$ is an element of a group, and therefore has a well-defined inverse $E(P)^{-1}$. It is of independent interest to ask whether, for a given operator $P$, the homology cylinder $E(P)^{-1}$ corresponds to a satellite operator, that is, if there exists some $\overline{P}\in\S$ such that $E(P)^{-1}=E(\overline{P})$ in $\S_*$ (we say that $\overline{P}$ is an inverse for $P$). In fact, if $P$ has an inverse $\overline{P}$, we may infer that $P$ gives a bijective function on $\C_*$ (Corollary~\ref{cor:givessurj}). In Section~\ref{sec:surjections}, we give a sufficient condition for a pattern to have an inverse, which allows us to construct a family of satellite operators which induce bijective functions on $\C_*$, as follows. Of course, it is easy to see that connected sum operators (of the form shown in Figure~\ref{fig:connected sum}) give bijective functions on $\C_*$ for each $*\in\{\ex,\top, R\}$. Consequently, when seeking examples of bijective satellite operators one should make sure that one avoids patterns that are equivalent to connected sum operators. 

\newtheorem*{thm:honestinverse}{Theorem~\ref{thm:honestinverse}}\begin{thm:honestinverse} Let $P\subseteq V=S^1\times D^2$ be a pattern with winding number $\pm 1$. If the meridian of $P$ is in the normal subgroup of $\pi_1(E(P))$ generated by the meridian of $V$, then $P$ is strong winding number $\pm 1$ and there exists another strong winding number $\pm 1$ pattern $\overline{P}$ such that $E(P)^{-1}=E(\overline{P})$ as homology cylinders.\end{thm:honestinverse}

\newtheorem*{cor:surjexample}{Corollary~\ref{cor:surjexample}}\begin{cor:surjexample} 
For each $m\geq 0$, the pattern $P_m\subseteq V(P_m)=S^1\times D^2$ shown in Figure~\ref{fig:surjex} gives a bijective map $P_m:\C_*\to \C_*$ for $*\in\{\ex,\top,\Z\}$. Moreover, each $P_m$ is distinct from all connected sum operators.\end{cor:surjexample}

\begin{figure}[b]
\includegraphics{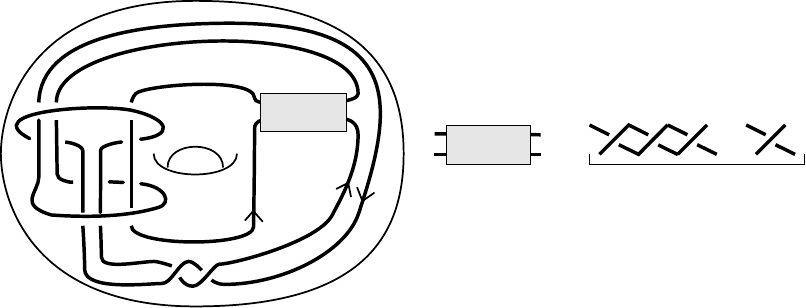}
\put(-2.1675,0.75){\tiny $2m+1$}
\put(-1.425,0.625){\tiny $2m+1$}
\put(-1,0.625){\tiny $=$}
\put(-0.85,0.5){\tiny $2m+1$ half-twists}
\caption{A class of bijective satellite operators $\{P_{m}\}_{m\geq 0}$ (see Theorem~\ref{cor:surjexample}).}\label{fig:surjex}
\end{figure}
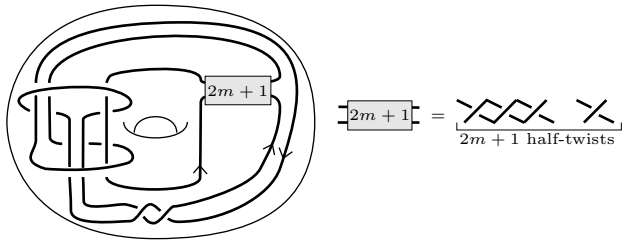

In our final application of the main theorem, we draw a connection between the surjectivity of satellite operators, specifically Akbulut's conjecture~\cite[Problem 1.45]{kirbylist} mentioned above, and a question of Matsumoto~\cite[Problem 1.30]{kirbylist} asking if every knot in a 3--manifold homology cobordant to $S^3$ is concordant in a homology cobordism to $S^3$ to a knot in $S^3$. That is, we show the following result.

\newtheorem*{prop:conjectures}{Proposition~\ref{prop:conjectures}}\begin{prop:conjectures} Let $*\in\{\ex,\top\}$. If there exists a pattern of strong winding number $\pm 1$ such that the induced function $P: \C_*\to\C_*$ is not surjective, then there exists a knot in a 3--manifold $M$ homology cobordant to $S^3$ which is not concordant, in the category $*$, to any knot in $S^3$ in any $*$--homology cobordism from $M$ to $S^3$. \end{prop:conjectures}

\subsection*{Outline} In Section~\ref{sec:background}, we give the relevant background and prove the main theorem. Section~\ref{sec:surjections} addresses surjectivity of satellite operators. In Section~\ref{sec:kirbyconnection} we prove Proposition~\ref{prop:conjectures}.

\subsection*{Remark}Shortly after a preprint of this paper was circulated, Adam Levine proved the existence of non-surjective satellite operators (see~\cite{Lev14}).

\subsection*{Acknowledgements} 
The authors would like to thank Tim Cochran for many helpful conversations throughout this project, during part of which he was the doctoral advisor of the second author.  

\section{Background and proof of the main theorem}\label{sec:background}

\subsection{Satellite operators}\label{sec:satellites}
For a pattern $P\subseteq V=S^1\times D^2$, let $E(P)$ denote the complement of a regular neighborhood of $P$ in $V$.  There are four important (oriented) curves on the boundary of $E(P)$:
\begin{enumerate}
\item $m(P)$, the \textit{meridian of the pattern}, 
\item $\ell(P)$, the \textit{longitude of the pattern}, 
\item $m(V) = m(V(P)) =1\times \bdry D^2$, the \textit{meridian of the solid torus} and 
\item $\ell(V) = \ell(V(P)) = S^1\times 1$, the \textit{longitude of the solid torus}.
\end{enumerate}  

Here $\ell(P)$ is the pushoff (unique up to isotopy) of $P$ in $V$ which is homologous in $E(P)$ to a multiple of $\ell(V)$. We say that $P$ has \textit{winding number} $w=w(P)\in \Z$ if $\ell(P)$ is homologous to $w\cdot \ell(V)$.  Consistent with this definition, on the torus $T=S^1\times S^1$, we will call the curve $\ell = S^1\times\{1\}$ the longitude of $T$, and the curve $m = \{1\}\times S^1$ the meridian of $T$. Let $\S$ denote the set of all satellite operators up to isotopy. 

\begin{remark}\label{rem:link}
Given a satellite operator $P\subseteq V$ one gets a knot $\widetilde P\subseteq S^3$ by adding a 2--handle to $V$ along $\ell(V)$, followed by a 3--handle along the resulting 2--sphere boundary.  Let $\eta$ be the image of $m(V)$ after this handle addition. It is straightforward to see that the map $P\mapsto (\widetilde{P}, \eta)$ gives a bijection from $\S$ to the set of (ordered, oriented) 2--component links in $S^3$ whose second component is unknotted, since given such a link we can remove a tubular neighborhood of the second component to get a knot (the first component) in a solid torus, namely a satellite operator.  The winding number of $P$ is the same as the linking number between $\widetilde P$ and $\eta$, and $E(P)$ is the complement of the link $(\widetilde{P}, \eta)$ in $S^3$; see Figure~\ref{fig:patterntolink}. \end{remark}

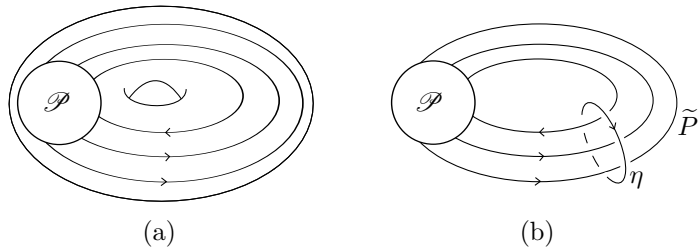
\begin{figure}[t]
\includegraphics[width=3.5in]{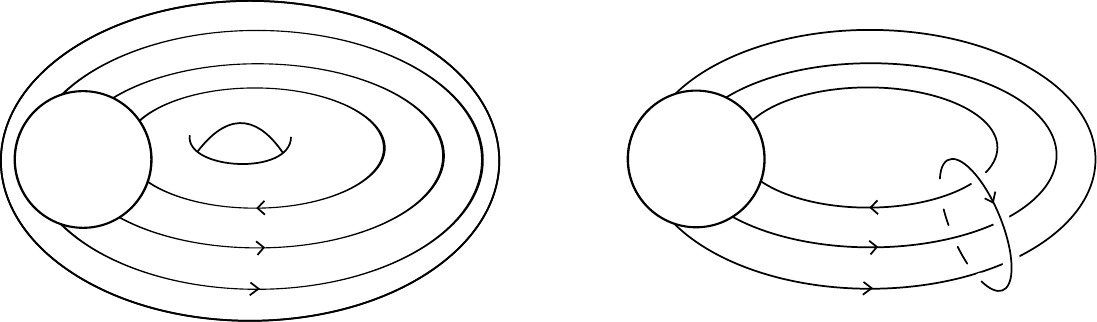}
\put(-3.325,0.475){$\mathscr{P}$}
\put(-1.375,0.475){$\mathscr{P}$}
\put(-0.25,0.1){$\eta$}
\put(-0,0.35){$\widetilde{P}$}
\put(-2.8,-0.2){(a)}
\put(-0.825,-0.2){(b)}
\caption{(a) A schematic picture of a satellite operator $P$. The circle containing $\mathscr{P}$ denotes a tangle. (b) The 2--component link corresponding to the satellite operator $P$. The circle containing $\mathscr{P}$ denotes the same tangle as in the previous panel.}\label{fig:patterntolink}
\end{figure}

The set $\S$ has a natural monoid structure. Given two satellite operators $P\subseteq V(P)$ and $Q\subseteq V(Q)$, where $V(P)$ and $V(Q)$ are standard solid tori, we construct the composed satellite operator $P\star Q$ as follows. Glue $E(Q)$ and $V(P)$ together by identifying $\bdry E(Q)$ with $\bdry V(P)$ via $m(Q)\sim m(V(P))$ and $\ell(Q)\sim\ell(V(P))$. The product, $P\star Q$ is the image of $Q$ after this identification and the resulting manifold is still a solid torus $V(P\star Q)$ (see Figure~\ref{fig:P*Q}). The operation $\star$ is clearly associative, i.e.\ $P\star (Q\star S) = (P\star Q)\star S$, and the monoid identity is given by the core of the solid torus, namely the \textit{trivial} satellite operator. 

\begin{figure}[b]
\begin{picture}(4.75,1)
\put(0,0.05){\includegraphics{pattern.pdf}}
\put(0.7,-0.1){$P$}
\put(1.65,0.05){\includegraphics{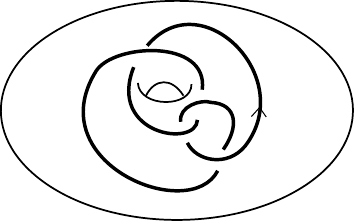}}
\put(2.25,-0.1){$Q$}
\put(3.25,0.05){\includegraphics{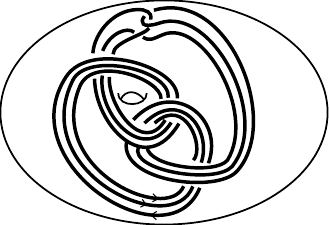}}
\put(3.725,-0.1){$P\star Q$}
\end{picture}
\caption{The monoid operation on patterns.}\label{fig:P*Q}
\end{figure}

Satellite operators act on knots in $S^3$ as we described in Figure~\ref{fig:satelliteoperation}. To obtain $P(K)$ from a pattern $P\subseteq V$ and a knot  $K\subseteq S^3$, start with the knot complement $E(K)$. The toral boundary contains the oriented curves $\ell(K)$, the \textit{longitude of $K$}, and $m(K)$, the \textit{meridian of $K$}. Glue in $V(P)$ by identifying $\ell(V)\sim \ell(K)$ and $m(V)\sim m(K)$. The resulting 3--manifold is $S^3$ and the image of $P$ is the \textit{satellite knot} $P(K)$.  For further details see~\cite[p. 111]{Ro90}.

Let $\K$ denote the set of knots in $S^3$ modulo isotopy. For patterns $P$ and $Q$ and a knot $K$, we easily see that $(P\star Q)(K) = P(Q(K))$.  Therefore, we have a monoid homomorphism $\S\to \maps(\K,\K)$, that is, a monoid action on the set of isotopy classes of knots in $S^3$. The following proposition shows that $\S$ is far from being a group under the operation $\star$. 

\begin{proposition}\label{prop:noinverse}The only element of $\S$ which has an inverse under the operation $\star$ is the trivial operator given by the core of the solid torus. \end{proposition}
\begin{proof}We see this using the notion of \textit{bridge index} $\br(K)$ of a knot $K$~\cite[p. 114]{Ro90}. Suppose $P\subseteq V$ has an inverse denoted $P^{-1}$. For any satellite knot $P(K)$, we know from \cite{Schu54} that $\br(P(K))\geq n\cdot\br(K)$ where $n$ is the \textit{geometric winding number} of $P$, i.e.\ the minimal (unsigned) number of intersections between $\ell(P)$ and a meridional disk of $V$. Therefore, for any operator $P$ and knot $K$, $\br(P(K))\geq \br(K)$ and so,
$$\br(K)=\br((P^{-1}\star P)(K))=\br(P^{-1}(P(K))) \geq \br(P(K))\geq \br(K).$$
Thus $\br(P(K))=\br(K)$ and $P$ has geometric winding number one. This implies that $P$ is a \textit{connected sum operator}, i.e.\ $P=Q_J$ for some knot $J$ as shown in Figure~\ref{fig:connected sum}. However, connected sum with a non-trivial knot cannot have an inverse due to the additivity of genus. Therefore, $P=Q_U$ where $U$ is the unknot. This completes the proof since $Q_U$ is the trivial satellite operator.\end{proof}

\begin{figure}[b]
\includegraphics{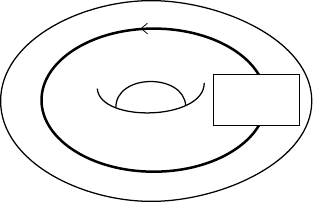}
\put(-0.28,0.375){$J$}
\caption{The connected sum operator $Q_J$ corresponding to the knot $J$. A strand going up through the box marked $J$ has the knot type of $J$. }\label{fig:connected sum}
\end{figure}

The following submonoids of $\S$ will be of particular interest in this paper.

\begin{definition}\label{def:Smonoids}Let $P$ be a pattern and $\Z\subseteq R\subseteq\Q$ a localization of $\Z$.
\begin{enumerate}
\item[(a)] $P$ is said to lie in $\S_R$ if $w(P)$ is invertible in $R$, that is, $\dfrac{1}{w(P)}\in R$.
\item[(b)]\label{def:SW1} $P$ is said to lie in $\Sstr$ (and have \textit{strong winding number $\pm 1$}) if $w(P)=\pm1$ and each of the sets $\{m(V),\ell(V)\}$ and $\{m(P),\ell(P)\}$ normally generates $\pi_1(E(P))$.
\end{enumerate}
\end{definition}

Recall from Remark~\ref{rem:link} that one can obtain a knot $\widetilde{P}$ from a pattern $P\subseteq V=S^1\times D^2$ by adding a 3--dimensional 2--handle to $V$ along $\ell(V)$ and then a 3--dimensional 3--handle along the resulting 2--sphere boundary (in fact $\widetilde{P}=P(U)$, where $U$ is the unknot). In~\cite{CDR14}, a pattern $P$ was said to have strong winding number $\pm 1$ if $w(P)=\pm1$ and $m(V)$ normally generates $\pi_1(S^3-\widetilde{P})$. Our definition is equivalent to the definition in~\cite{CDR14} as follows. 

\begin{proposition}\label{Prop: strong winding number 1}For a winding number $\pm 1$ satellite operator $P\subseteq V=S^1\times D^2$, $m(V)$ normally generates $\pi_1(S^3-\widetilde{P})$ if and only if each of the sets $\{m(V),\ell(V)\}$ and $\{m(P),\ell(P)\}$ normally generates $\pi_1(E(P))$.\end{proposition}

\begin{proof}From the definition of $\widetilde{P}$, it is clear that $\pi_1(S^3-\widetilde{P}) = \pi_1(E(P))/\langle\langle\ell(V)\rangle\rangle$, where $\langle\langle\ell(V)\rangle\rangle$ is the normal subgroup generated by $\ell(V)$ in $\pi_1(E(P))$. The backward direction follows immediately, since if $\{m(V),\ell(V)\}$ normally generates $\pi_1(E(P))$, the group $\pi_1(S^3-\widetilde{P})=\pi_1(E(P))/\langle\langle\ell(V)\rangle\rangle$ must be normally generated by $m(V)$. For the forward direction, note that if we assume that $m(V)$ normally generates $\pi_1(S^3-\widetilde{P}) = \pi_1(E(P))/\langle\langle\ell(V)\rangle\rangle$, then $\{m(V),\ell(V)\}$ normally generates $\pi_1(E(P))$. In order to see that $\{m(P),\ell(P)\}$ also normally generates $\pi_1(E(P))$, notice that $V$ is obtained from $E(P)$ by adding a 2--handle along $m(P)$ (followed by a 3--handle) so that 
$$\Z\cong\pi_1(V)= \pi_1(E(P))/\langle\langle m(P)\rangle\rangle.$$
Since $P$ is winding number $\pm 1$, $\pi_1(V)$ is generated by $\ell(P)$ which is homotopic in $V$ to $\ell(V)$.  Thus $\{m(P),\ell(P)\}$ normally generates $\pi_1(E(P))$. \end{proof}

\subsection{Homology cobordism classes of homology cylinders}\label{sec:homology cylinder}

In \cite{Lev01} Levine defined the group of integral homology cylinders over a surface, with the goal of producing an enlargement of the mapping class group. For completeness, and since we require slight variants and generalizations, we recall the definitions below.

\begin{definition}[\cite{Lev01}]
Let $T=S^1\times S^1$ be the torus and $R$ a localization of $\Z$.  
An \textit{$R$--homology cylinder on $T$}, or an \textit{$R$--cylinder}, is a triple $(M,i_+,i_-)$ where 
\begin{itemize}
\item $M$ is an oriented, compact, connected 3--manifold
\item For $\epsilon=\pm1$, $i_\epsilon:T\to \bdry M$ is an embedding, with $\bdry M$ = $i_+(T) \sqcup i_-(T)$
\item  $i_+$ is orientation-preserving and $i_-$ is orientation-reversing
\item $(i_\epsilon)_*:H_*(T;R)\to H_*(M;R)$ is an isomorphism.
\end{itemize}
For two $R$--homology cylinders,  $(M,i_+,i_-)$ and $(N,j_+,j_-)$, we say that $(M,i_+,i_-) = (N,j_+,j_-)$ if there is a homeomorphism $\phi:M\to N$ such that $\phi\circ i_\epsilon = j_\epsilon$ for $\epsilon\in\{\pm1\}$.

A $\Z$--cylinder $(M,i_+,i_-)$ is called a \textit{strong cylinder} if $\pi_1(V)$ is normally generated by each of $\im(i_+)_*$ and $\im(i_-)_*$. Let $H_R$ denote the set of all $R$--cylinders and $H_\str$ denote the set of all strong cylinders.\end{definition}

For $*\in\{\str, R\}$, there is a monoid operation on $H_*$  given by stacking:  
$$(M,i_+,i_-)\star(N,j_+,j_-) = \left(\bigslant{M \sqcup N}{i_+(x)\sim j_-(x), \forall x\in T},j_+,i_-\right)$$
The identity element with respect to $\star$ is given by $(T\times [0,1],\id\times\{1\},\id\times\{0\})$. 

\begin{definition}[\cite{Lev01}]Two $R$--cylinders $(M,i_+,i_-)$ and $(N,j_+,j_-)$ are said to be \textit{$R$--cobordant} if there is a smooth 4--manifold $W$ with 
$$\bdry W = \bigslant{M\sqcup -N}{i_+(x)=j_+(x),i_-(x)=j_-(x), \forall x\in T},$$
such that $H_*(M;R) \to H_*(W;R)$ and $H_*(N;R) \to H_*(W;R)$ are isomorphisms. This is equivalent to requiring that the compositions 
$$H_*(T;R)\xrightarrow{(i_\epsilon)_*} H_*(M;R)\to H_*(W;R)$$
$$H_*(T;R)\xrightarrow{(i_\epsilon)_*} H_*(N;R)\to H_*(W;R)$$
are isomorphisms for each $\epsilon\in\{\pm1\}$. Such a $W$ is called an \textit{$R$--cobordism}. $\H_R$ denotes the set of all $R$--cobordism classes of $R$--cylinders.

Two strong cylinders $(M,i_+,i_-)$ and $(N,j_+,j_-)$ are said to be \textit{strongly cobordant} if there exists a $\Z$--cobordism $W$ between $V$ and $U$ such that $\pi_1(W)$ is normally generated by each of $\pi_1(V)$ and $\pi_1(U)$. Such a $W$ is called a \textit{strong cobordism}. This is equivalent to requiring that the images of 
$$\pi_1(T)\xrightarrow{(i_\epsilon)_*} \pi_1(M)\to \pi_1(W)$$
$$\pi_1(T)\xrightarrow{(i_\epsilon)_*} \pi_1(N)\to \pi_1(W)$$
(individually) normally generate $\pi_1(W)$ for each $\epsilon\in\{\pm1\}$. $\H_\ex$ denotes the set of all strong cobordism classes of strong cylinders.  

In the latter definition if the manifold $W$ is not required to be smooth, we say $(M,i_+,i_-)$ and $(N,j_+,j_-)$ are \textit{strongly topologically cobordant}. $\H_\top$ denotes the set of strong topological cobordism classes of strong cylinders. 
\end{definition}

In \cite{Lev01}, Levine proves that the binary operation $\star$ on $H_\Z$ is well-defined on $\H_\Z$ and endows $\H_\Z$ with the structure of a group.  Indeed, $N\times[0,1]$ can be seen to be a cobordism  between $(N,i_+,i_-)\star(-N,i_-,i_+)$ and the identity element $(T\times[0,1],\id\times\{0\},\id\times\{1\})$.  Thus, the  \textit{inverse} of $(N,i_+,i_-)$ in $\H_\Z$ is $(-N,i_-,i_+)$, where $-N$ denotes the orientation-reverse of $N$. Since $N\times[0,1]$ is a smooth $R$--cobordism
 when $N$ is an $R$--cylinder,  and  $\H_R$ is a group for all localizations $R$ of $\Z$. Similarly, if $N$ is a strong cylinder it is easy to see that $N\times[0,1]$ is a smooth and topological strong cobordism.  Thus, $\H_\ex$ and $\H_\top$ are also groups.

\subsection{Patterns as homology cylinders}\label{sec:Shat}

For any pattern $P\in \S_R$, the exterior of $P$, $E(P)$, can be seen to be an $R$--homology cylinder in a natural way.  Let $i_+$ be the identification $S^1\times S^1\to \bdry V$ sending $m\mapsto m(V)$ and $\ell\mapsto \ell(V)$.  Similarly let  $i_-$ be the identification of the boundary of a tubular neighborhood of $P$ with $S^1\times S^1$ which sends $\ell\mapsto \ell(P)$ and $m\mapsto m(P)$.  A Mayer--Vietoris argument easily reveals that $(E(P), i_+, i_-)$ is an $R$--cylinder. It follows immediately from our definitions that if $P\in \S_\str$, then $(E(P), i_+, i_-)\in H_\str$. Henceforth we will often abuse notation by letting $E(P)$ denote the $*$--cylinder $(E(P),i_+,i_-)$, where $*\in\{\str,R\}$. For each value of $*$, we have a map 
\begin{align*}
E:\S_*&\to H_*\\
P&\mapsto E(P)
\end{align*}
which is easily seen to be a monoid homomorphism.

Note that if $P$ is a pattern of winding number $w\neq 0$, then in $E(P)$, the curve $\ell(P)$ is homologous to $w\cdot\ell(V)$ and the curve $m(P)$ is homologous to $(1/w)\cdot m(V)$.  Thus, with respect to the basis $\{\ell,m\}$ for $H_1(S^1\times S^1)$, the composition $\left((i_+)_*^{-1}\circ (i_-)_*\right)$ is given by the matrix 
$
\left[
\begin{array}{cc}
w&0\\0&1/w
\end{array}
\right].
$
This motivates the following definition.

\begin{definition}Let $R$ be a localization of $\Z$. We define $\Shat^0_R\subseteq H_R$ to be the submonoid of all $R$--cylinders $(M,i_+,i_-)$ for which the map $(i_+)_*^{-1}\circ (i_-)_*:H_1(T;R)\to H_1(T;R)$  has determinant one and is diagonal with respect to the basis $\{\ell, m\}$. Similarly, $\Shat^0_\str\subseteq H_\str$ is the submonoid of all strong cylinders $(M,i_+,i_-)$ for which the map $(i_+)_*^{-1}\circ (i_-)_*:H_1(T;\Z)\to H_1(T;\Z)$ is $\pm\id$. Elements of $\Shat^0_*$ will be called \textit{generalized $*$--satellite operators}, for $*\in\{\str, R\}$.
\end{definition}

As a result of our previous discussion, we see the following result.

\begin{proposition}\label{prop:monoidmorphism1}For each $* \in\{\str, R\}$ there is a monoid homomorphism $$E:\S_*\to \Shat^0_*.$$\end{proposition}

Moreover, the submonoids $\Shat^0_*$ of $H_*$ are closed under the map $(N,i_+,i_-)\mapsto (-N, i_-, i_+)$ sending a $*$--cylinder to its inverse in $\H_*$. Therefore, we define 
\begin{align*}
\Shat_\ex :=& \,\,\bigslant{\Shat^0_\str}{\text{strong cobordism}} &\Shat_\top :=& \,\,\bigslant{\Shat^0_\str}{\text{strong topological cobordism}}\\
\Shat_R :=& \,\,\bigslant{\Shat^0_R}{R\text{--cobordism}} &
\end{align*}
and see that $\Shat_\ex$, $\Shat_\top$, and $\Shat_R$ are subgroups of $\H_\ex$, $\H_\top$, and $\H_R$ respectively.

From the above, using the monoid morphisms $E$ from Proposition~\ref{prop:monoidmorphism1}, we obtain the following proposition.

\begin{proposition}\label{prop:homomorphism}There are monoid homomorphisms, 
$$E:\S_R\to \Shat_R,\, E:\S_\str\to \Shat_\ex\text{, and } \S_\str\to \Shat_\top.$$
\end{proposition}

\subsection{Generalizations of  knot concordance}\label{sec:generalizedconcordance}

\begin{definition}\label{def:genconcdefn}
Let $K$ and $J$ be knots in the $\Z$--homology spheres $X$ and $Y$ respectively. $(K,X)$ and $(J,\,Y)$ are called \textit{exotically concordant} (resp.\ \textit{topologically concordant}) if there is a smooth (resp.\ topological) $\Z$--homology cobordism $W$ from $X$ to $Y$ with $\pi_1(W)$ normally generated by the images of each of $\pi_1(X)$ and $\pi_1(Y)$ and in which $K$ and $J$ cobound a smooth (resp.\ locally flat) annulus.  The set of all knots in $\Z$--homology spheres modulo exotic concordance (resp.\ topological concordance) is denoted by $\Chat_\ex$ (resp.\ $\Chat_\top$).

Now suppose that $R$ is a localization of $\Z$, and $K$ and $J$ are knots in the $R$--homology spheres $X$ and $Y$. $(K,X)$ and $(J,\,Y)$ are called \textit{$R$--concordant} if there is a smooth $R$--homology cobordism $W$ from $X$ to $Y$ in which $K$ and $J$ cobound a smooth annulus. We denote by $\Chat_R$ the set of all knots in $R$--homology spheres modulo $R$--concordance. \end{definition}

\begin{proposition}\label{prop:enlarge} The map 
\begin{align*}
\Psi:\C_*&\to \Chat_*\\
[K]&\mapsto[(K, S^3)]
\end{align*}
is well-defined and injective, for each value of $*\in\{ \ex,\top, R \}$.\end{proposition}

\begin{proof}Let $* = \ex$.  If $K$ is exotically concordant to $J$, then $K$ and $J$ cobound a smooth annulus in $W:=S^3\times[0,1]$ with a possibly exotic smooth structure. Therefore, $W$ is a homology cobordism from $S^3$ to itself and since $\pi_1(W)=0$, $(K,S^3)=(J,S^3)$ in $\Chat_\ex$. Hence, the map $\Psi$ is well-defined.

Suppose $(K,S^3)=(J,S^3)$ in $\Chat_\ex$. Then $K$ and $J$ cobound a smooth annulus in a smooth homology cobordism $W$ from $S^3$ to itself with $\pi_1(W)$ normally generated by $\pi_1(S^3)=0$.  Thus, $W$ is simply connected, and by Freedman's proof of the topological 4--dimensional Poincar\'e Conjecture~\cite{Free82}, $W$ is homeomorphic to $S^3\times[0,1]$ (but not necessarily diffeomorphic). Thus, $K$ is exotically concordant to $J$.

The proof in the cases $*=\top$ or $R$ is similar. \end{proof}

\begin{remark}\label{rem:concthencob}We noted in Remark~\ref{rem:link} that satellite operators $P_i\subseteq V_i=S^1\times D^2$ ($i=0,1$)  in $\S_\str$ can be uniquely represented by the 2--component links $(\widetilde{P_i},\eta_i)$. In fact if the links $(\widetilde{P_0},\eta_0)$ and $(\widetilde{P_1},\eta_1)$ are exotically (resp.\ topologically) concordant, then the strong homology cylinders  $E(P_0)$ and $E(P_1)$ are strongly (resp.\ strongly topologically) cobordant. This is seen by cutting out a regular neighborhood of the concordance between the two links in $S^3\times[0,1]$. Similarly, suppose the operators $P_i$ are in $\S_R$ for $R$ some localization of $\Z$;  if the links $(\widetilde{P_0},\eta_0)$ and $(\widetilde{P_1},\eta_1)$ are $R$--concordant, then $E(P_0)$ and $E(P_1)$ are $R$--cobordant. 
\end{remark}

\begin{remark}\label{rem:cobthenconc} 
A generalized $*$--satellite operator $(M, i_+, i_-)$ yields a link $(\widetilde{P_M}, \eta_M)$ in a 3--manifold $\widehat{M}$, as follows. We obtain $\widehat{M}$ by attaching 2--handles to $M$ along $i_-(m)$ and $i_+(\ell)$, followed by 3--handles along the two resulting sphere boundary components. Let $\widetilde{P_M}$ denote the image of $i_-(\ell)$ in $\widehat{M}$ and $\eta_M$ the image of $i_+(m)$ in $\widehat{M}$. If $*=R$ then $\widehat M$ is an $R$--homology sphere, whereas if $*=\str$ then $\widehat{M}$ is a $\Z$--homology sphere. It is straightforward to see that if two $R$--cylinders $(M,i_+, i_-)$ and $(N, j_+, j_-)$ are $R$--cobordant then $(\widetilde{P_M}, \eta_M)$ and $(\widetilde{P_N}, \eta_N)$ are concordant as links in an $R$--homology cobordism from $\widehat M$ to $\widehat N$.  If $M$ and $N$ are strong homology cylinders which are strongly cobordant then the links $(\widetilde{P_M}, \eta_M)$ and $(\widetilde{P_N}, \eta_N)$ are concordant in a homology cobordism from $\widehat{M}$ to $\widehat{N}$ whose fundamental group is normally generated by each of $\pi_1(\widehat{M})$ and $\pi_1(\widehat{N})$.  If the strong homology cylinders are merely strongly topologically cobordant, then the link concordance is topological. \end{remark}
\subsection{Generalized satellite operators act on knots in homology spheres.}  \label{sec:action}

\begin{proposition}\label{prop:commutes}
The monoid $\Shat_R^0$ acts on the set of knots in $R$--homology spheres. For the map $E:\S_R\rightarrow\Shat_R^0$, if $P\in \S_R$ and $K\subseteq S^3$ is a knot, $(P(K),S^3)$ is isotopic to $(E(P))(K,S^3)$. \end{proposition} 

Note that since $\Shat_\str^0\subseteq \Shat_\Z^0$, the above says that $\Shat_\str^0$ acts on the set of knots in $\Z$-homology spheres.

\begin{proof}[Proof of Proposition~\ref{prop:commutes}]Let $K$ be a knot in the $R$--homology sphere $Y$, and $(M,i_+,i_-)$ be an $R$--homology cylinder. Construct a 3--manifold $Y'$ as follows---start with $M$, glue on a solid torus along $i_-(T)$ such that $i_-(m)$ bounds a disk, and glue in $Y-N(K)$ such that $i_+(\ell)\sim \ell(K)$ and $i_+(m)\sim m(K)$. Therefore, 
$$Y':= S^1\times D^2 \underset{\bdry D^2\sim i_-(m)}{\cup} M \underset{i_+(\ell) \sim \ell(K)}{\underset{i_+(m) \sim m(K)}{\cup}} Y-K.$$

It is easy to check that $Y'$ is an $R$--homology sphere when $(V,i_+,i_-)\in\Shat_R^0$. Let $K'$ be the core of the solid torus $S^1\times D^2$ in this decomposition. The above construction gives the desired action on knots in $R$--homology spheres, that is,
$$(M,i_+,i_-)\cdot(K,Y):=(K',Y')$$
It is straightforward to see that for all homology cylinders $M$ and $N$ and any pair $(K,Y)$ as above, $(M\star N)\cdot (K,Y) = M\cdot (N\cdot (K,Y))$. Therefore we have a monoid action. The pairs $(P(K),S^3)$ and $(E(P))\cdot (K,S^3)$ are isotopic since the gluing instructions given above are identical to those in the classical satellite construction. \end{proof}

\begin{proposition}\label{prop:acts} Let $R$ be a localization of $\Z$ and $* \in\{\ex,\top, R\}$. The monoid action of Proposition~\ref{prop:commutes} descends to a group action by $\Shat_*$ on $\Chat_*$ for each choice of $* \in\{\ex,\top, R\}$.\end{proposition}

\begin{proof}Let $(M, i_+, i_-), (N,j_+,j_-)\in \S_*$ be generalized $*$--satellite operators, for $*\in\{\str,R\}$, and $K$ and $J$ be knots in the manifolds $Y$ and $X$ respectively. From the proof of Proposition~\ref{prop:commutes}, we know that $(M,i_+,i_-)\cdot (K,Y)$ and $(N,j_+,j_-)\cdot (J,X)$ are knots in the 3--manifolds $Y' = S^1\times D^2\,\cup\, M\cup E(K)$ and $X' = S^1\times D^2\,\cup\, N\cup E(J)$, where, in particular, the resulting knots are given by the cores of the $S^1\times D^2$--pieces.  

Suppose $(K,Y)$ and $(J,X)$ are $R$--concordant, i.e.\ $(K,Y)=(J,X)$ in $\Chat_R$, and $(M,i_+,i_-)$ and $(N,j_+,j_-)$ are $R$--cobordant, i.e.\ $(M,i_+,i_-)=(N,j_+,j_-)$ in $\Shat_R$. Then there is an $R$--cobordism $U_0$ between $M$ and $N$, and an $R$--concordance $C$ from $K$ to $J$ in some 4--manifold; let $E(C)$ be the complement of $C$.

The gluing instructions used to build $X'$ and $Y'$ extend to gluing instructions for a 4--manifold
$$U=S^1\times D^2\times[0,1]\,\,\cup\,\, U_0\cup E(C)$$
with $\bdry U=Y'\sqcup -X'$. Since each of the pieces of $U$ is smooth, $U$ is smooth as well. Moreover, since each of the pieces of $U$ is an $R$--homology cobordism, it follows from a Mayer--Vietoris argument that $U$ is an $R$--homology cobordism as well, where the core of $S^1\times D^2 \times[0,1]$ is a smooth annulus cobounded by the two resulting knots. This completes the proof in the case that $*=R$. 

For the $*=\ex$ case, we only need to show that the additional condition on fundamental groups is satisfied when the spaces used above are in $\Shat_\ex$ and $\Chat_\ex$. This can be seen using two successive Seifert--van Kampen arguments since the fundamental group of each piece of $U$ is normally generated by each of its boundary components. The last remaining case, $*=\top$, follows from the various arguments above, with the additional trivial observation that if the pieces of $U$ are merely topological, the core of $S^1\times D^2 \times[0,1]$ is a locally flat annulus. \end{proof}

We combine the results and definitions of this section to give the main theorem. 

\begin{theorem_main}
Let $R$ be a localization of $\,\Z$. For the maps $\Psi$ from Proposition~\ref{prop:enlarge}, the monoid morphisms  $E:\S_\str\to \Shat_\ex$, $\S_\str\to \Shat_\top$ and $\S_R\to \Shat_R$ from Proposition~\ref{prop:homomorphism}, and any $P\in \S_\str$, and $Q\in\S_R$, the following diagrams commute.
\begin{equation}\label{diag:maindiagram}
\begin{tikzpicture}[baseline=(current  bounding  box.center)]
  \matrix (m) [matrix of math nodes,row sep=3em,column sep=4em,minimum width=2em]
  {
     \C_\ex & \C_\ex \\
     \Chat_\ex & \Chat_\ex \\};
  \path[-stealth]
    (m-1-1) edge node [above] {$P$} (m-1-2)
    (m-2-1) edge node [above] {$E(P)$} (m-2-2);
  \path[-stealth,right hook->]
     (m-1-1) edge node [left] {$\Psi$} (m-2-1)
     (m-1-2) edge node [right] {$\Psi$} (m-2-2);
     
\end{tikzpicture}\hspace{20pt}
\begin{tikzpicture}[baseline=(current  bounding  box.center)]
  \matrix (m) [matrix of math nodes,row sep=3em,column sep=4em,minimum width=2em]
  {
     \C_\top & \C_\top \\
     \Chat_{\smash{\top}} & \Chat_{\smash{\top}} \\};
  \path[-stealth]
    (m-1-1) edge node [above] {$P$} (m-1-2)
    (m-2-1) edge node [above] {$E(P)$} (m-2-2);
  \path[-stealth,right hook->]
    (m-1-1) edge node [left] {$\Psi$} (m-2-1)  
    (m-1-2) edge node [right] {$\Psi$} (m-2-2);
\end{tikzpicture}\hspace{20pt}
\begin{tikzpicture}[baseline=(current  bounding  box.center)]
  \matrix (m) [matrix of math nodes,row sep=3em,column sep=4em,minimum width=2em]
  {
     \C_R & \C_R \\
     \Chat_R & \Chat_R \\};
  \path[-stealth]
    (m-1-1) edge node [above] {$Q$} (m-1-2)
    (m-2-1) edge node [above] {$E(Q)$} (m-2-2);
  \path[-stealth,right hook->]
     (m-1-1) edge node [left] {$\Psi$} (m-2-1)
     (m-1-2) edge node [right] {$\Psi$} (m-2-2);
\end{tikzpicture}
\end{equation}
\end{theorem_main}

\begin{proof}The result follows from Propositions \ref{prop:homomorphism}, \ref{prop:enlarge}, and \ref{prop:commutes}.\end{proof}

As an immediate corollary of the main theorem we recover the following result from \cite{CDR14}.
\begin{corollary}[Theorem 5.1 of \cite{CDR14}]\label{cor:injectivecorollary}Let $P$ be a pattern. If $P$ has winding number $n\neq 0$ then $P:\C_{\Z[1/n]}\to\C_{\Z[1/n]}$ is injective. If $P$ has strong winding number $\pm 1$ then $P:\C_{\ex}\to\C_{\ex}$ and $P:\C_{\top}\to\C_{\top}$ are injective. \end{corollary}

\begin{proof}The proof is a straightforward diagram chase. Let  $*\in\{\ex, \top, \Z[\frac{1}{n}]\}$. Suppose that $P(K)$ is concordant to $P(J)$ in the $*$--category, i.e.\ $P(K)=P(J)$ in $\C_*$. Then $\Psi(P(K)) = \Psi(P(J))$. Since the diagrams in~\eqref{diag:maindiagram} commute, we see that $(E(P))(\Psi(K)) = (E(P))(\Psi(J))$. Since $E(P)\in \Shat_*$ is an element of a group which acts on $\Chat_*$, it has an inverse. Therefore, the map $E(P)$ is bijective and in particular injective.  Thus, $\Psi(K) = \Psi(J)$. But $\Psi$ is also injective and therefore we conclude that $K=J$ in $\C_*$ as needed.\end{proof}

\section{Surjectivity of satellite operators}\label{sec:surjections}

Since satellite operators have now been shown to be injective in several categories (in Section~\ref{sec:background} as well as in \cite{CDR14}), it is natural to ask whether there exists a satellite operator $P$ such that the map $P:\C_*\to\C_*$ is \textit{surjective}, for $*\in \{\ex,\top,R\}$ for $R$ a localization of $\Z$. The following proposition shows that only patterns of winding number $\pm 1$ may be surjective. 

\begin{proposition}\label{prop:n>1}Let $P$ be a satellite operator with winding number $n\neq \pm 1$. The function $P:\C_*\to\C_*$ is not surjective for any $*\in\{\ex,\top, R\}$, where $R$ is a localization of $\Z$. \end{proposition}
\begin{proof}We know from \cite{Li77, LivM85} that for any knot $K$,
\begin{equation}\label{eqn:signatureobstruction}\sigma(P(K),\omega)=\sigma(P(U),\omega)+\sigma(K,\omega^n)\end{equation} 
where $U$ is the unknot, and $\sigma(\cdot,\omega)$ denotes the Levine-Tristram signature at $\omega\in\mathbb{C}$, $\lvert\omega\rvert=1$. For a fixed $P$ this imposes restrictions on the signature function of $P(K)$, as follows. Let $J$ be a knot for which $\sigma(J,\omega)$ is not of the form $g(\omega^n)$ for any function on $g$ on $S^1$, for example, the right-handed trefoil knot.  Then 
$\sigma(P(U)\#J, \omega) = \sigma(P(U), \omega) + \sigma(J, \omega)$ cannot be of the form prescribed to $\sigma(P(K),\omega)$ in equation~\eqref{eqn:signatureobstruction}. Therefore, $P(U)\#J$ is not in the image of $P$; the result follows since the signature function is an invariant of rational concordance.\end{proof}

As a result, we mostly restrict ourself henceforth to satellite operators in $\S_\Z$ and $\S_\str$. Of course, connected sum operators, i.e.\ satellite operators of the form $Q_J$ shown in Figure~\ref{fig:connected sum}, are clearly surjective. We say that a winding number $\pm 1$ satellite operator $P$ is \textit{non-trivial} if it is distinct as an element of $\Shat_\Z$ from the connected sum operators, $Q_J$ for all knots $J$. 

First we note we have a characterization of surjective satellite operators as follows. 

\begin{proposition}\label{prop:surjectivecharacterization}The satellite operator $P\in\S_\str$ gives a surjective map $P:\C_*\to\C_*$  for $*\in\{\ex, \top\}$ if and only if $E(P)^{-1}(\Psi(\C_*))\subseteq \Psi(\C_*)$, where $E(P)^{-1}$ is the inverse of the homology cylinder $E(P)\in\Shat_*$. Similarly, $P\in\S_\Z$ gives a surjective map $P:\C_\Z\to\C_\Z$ if and only if $E(P)^{-1}(\Psi(\C_\Z))\subseteq \Psi(\C_\Z)$\end{proposition}

\begin{proof}The key observation here is that since $E(P)$ acts via a group action, it must give a bijection on $\Chat_*$ for each $*\in\{\ex,\top,\Z\}$. Therefore, by the commutativity of the diagrams in~\eqref{diag:maindiagram} and the injectivity of $\Psi$, we see that $P:\C_*\to\C_*$ is surjective if and only if $E(P)^{-1}(\Psi(\C_*))=\Psi(\C_*)$. However, we know that $E(P)(\Psi(\C_*))\subseteq \Psi(\C_*)$ for all $P$. \end{proof}

It is worth noting that one way to guarantee that $E(P)^{-1}(\Psi(\C_*))\subseteq \Psi(\C_*)$, 
 for $*\in\{\ex,\top\}$, is for $E(P)^{-1}$ to be the image under $E:\S_\str\to\Shat_*$ of some $\overline{P}\in\S_\str$, since as we saw in the proof above, $E(\overline{P})(\Psi(\C_*))\subseteq\Psi(\C_*)$ for all $\overline{P}$. This holds for connected sum operators as shown below. 

\begin{proposition}For $*\in\{\top, \ex, R\}$ and any knot $J$ in $S^3$, $E(Q_J)^{-1}=E(Q_{-J})$ in $\Shat_*$. \end{proposition}

\begin{proof} To prove the result is suffices to find a strong cobordism from $E(Q_J)\star E(Q_{-J})$ to the identity element $(T\times[0,1], \id\times\{0\}, \id\times\{1\})$.  Since $E:\S_*\to \Shat_*$ is a homomorphism $E(Q_J)\star E(Q_{-J}) = E(Q_J\star Q_{-J})$.  Finally, it is easy to see from the definition of multiplication in $\S_*$ that $Q_J\star Q_{-J} = Q_{J\#-J}$.  

As a 3--manifold $E(Q_{J\#-J})$ is diffeomorphic the the complement in $S^3$ of the 2--component link $L$ consisting of $J\#-J$ and a meridian $\mu$ for ${J\#-J}$.  The diffeomorphism sends the longitude and meridian of $J\#-J$ to the longitude and meridian of $Q_{J\#-J}$ respectively, and the meridian and longitude of $\mu$ to the longitude and meridian of the solid torus $\ell(V)$ and $m(V)$ respectively.  

Finally, since $J\#-J$ is slice, the link $(J\# -J) \sqcup \mu$ is concordant to the Hopf link, whose exterior is diffeomorphic to $T\times[0,1]$.  It is straightforward to check that the complement of the concordance provides a $*$--cobordism between $E(Q_J\star Q_{-J})$ and the identity element $(T\times[0,1], \id\times\{0\}, \id\times\{1\})$.  \end{proof}

In fact, there exist non-trivial satellite operators $P$ with winding number $\pm 1$ such that $E(P)^{-1}=E(\overline{P})$ for some $\overline{P}\in\S_\Z$, as we see below. 

\begin{theorem}\label{thm:honestinverse}Let $P\subseteq V=S^1\times D^2$ be in $\S_\Z$. If $m(P)$ is in the normal subgroup of $\pi_1(E(P))$ generated by $m(V)$ then $P$ is strong winding number $\pm 1$ and there exists another strong winding number one pattern $\overline{P}$ such that $E(P)^{-1}=E(\overline{P})$ as homology cylinders.\end{theorem}

\begin{proof} We see that $P$ is strong winding number $\pm 1$ by Proposition~\ref{Prop: strong winding number 1}.   Indeed, in order to construct $S^3-\widetilde P$ from $P$, a 2-handle is added to $\ell(V)$.  Thus, $\pi_1(E(P))\to \pi_1(S^3-\widetilde P)$ is surjective.  By assumption, $m(P)$ is in the normal subgroup generated by $m(V)$ in $\pi_1(E(P))$.  Since $\pi_1(E(P))\to \pi_1(S^3-\widetilde P)$ is surjective, $m(P)$ is in the normal subgroup generated by $m(V) = \eta$ in $\pi_1(S^3-\widetilde{P})$.  Since $\widetilde{P}$ is a knot in $S^3$, $\pi_1(S^3-\widetilde{P})$ is normally generated by $m(P)$. Since $m(P)\in \langle\langle \eta \rangle\rangle$, it follows that $\eta$ normally generates $\pi_1(S^3-\widetilde{P})$.  Proposition~\ref{Prop: strong winding number 1} now concludes that $P$ is strong winding number $\pm1$.  

Note that $\pi_1(E(P))/\langle\langle m(P)\rangle\rangle\cong \pi_1(V)\cong\Z$ since the solid torus $V$ is obtained from $E(P)$ by adding a 2--handle to the meridian of $P$ and then a 3--handle.  Additionally, $m(V)$ is nullhomotopic in $V$ so that $m(V) = 0$ in $\pi_1(E(P))/\langle\langle m(P)\rangle\rangle$ and   $\langle\langle m(V)\rangle\rangle\subseteq \langle\langle m(P)\rangle\rangle$. By assumption, $m(P)\in \langle\langle m(V)\rangle\rangle$ so that we conclude that  $\langle\langle m(P) \rangle\rangle = \langle\langle m(V)\rangle\rangle$.  Therefore, $\pi_1(E(P)) / \langle\langle m(P)\rangle\rangle = \pi_1(E(P)) / \langle\langle m(V)\rangle\rangle \cong \Z$. 

Now consider $E(P)^{-1}$. By definition, $E(P)^{-1}=(-E(P),i_-,i_+)$. Perform a Dehn filling on $-E(P)$ along $m(V)$ to obtain a manifold $X$. By the preceding paragraph,  $\pi_1(X)\cong \Z$ and therefore, since $\partial X$ has no $S^2$ components, $X$ is diffeomorphic to the solid torus~\cite[Theorem 5.2]{Hemp04}.  Since $m(P)$ must be mapped to a curve which is null homologous, we see that $m(P)\mapsto 1\times \partial D^2$.  By performing meridional twists if necessary, we may assume that  $\ell(P)\mapsto S^1\times 1$. Then, by definition, if we denote by $\overline{P}$ the image of $\ell(V)$ in $X\cong S^1\times D^2$, we see that $E(\overline{P})=(-E(P),i_-,i_+)$. Since $\ell(V)$ is homologous to $\ell(P)$ in $E(P)$ (since $P\in\S_\Z$),  $w(\overline{P}) = w(P)=1$.  Since $P$ is strong winding number $\pm 1$ each of the sets $\{m(P), \ell(P)\}$ and $\{m(V(P)), \ell(V(P))\}$ normally generate $\pi_1(E(P))$.  But these sets of curves are respectively the same as $\{m(V(\overline P)), \ell(V(\overline P))\}$ and $\{m(\overline P), \ell(\overline P)\}$.  It follows that $\overline{P}$ is strong winding number $\pm 1$ as well. \end{proof}

Under the assumptions of the above theorem, the satellite operator $P$ has an inverse $\overline{P}$ which is also a satellite operator. A close reading of the proof of the theorem reveals how to draw a picture of the latter given the former. In fact, it is easier to see how to draw a picture of the corresponding 2--component link $\overline{L}$ (see Remark~\ref{rem:link} and Figure~\ref{fig:patterntolink}); recall that given such a link $\overline{L}$ we can recover the satellite operator $\overline{P}$ by removing a tubular neighborhood of the second component of $\overline{L}$ from $S^3$. 

Start with the 2--component link $L$ corresponding to the given satellite operator $P$. The manifold $E(P)$ is exactly the complement of this link in $S^3$. A key observation in the above proof is that the manifold obtained by performing a Dehn filling of $-(S^3-L)$ along the second component of $L$ is homeomorphic to a solid torus, via a homemorphism taking the first component of $L$ to the longitude of the solid torus. Of course, if we were to perform a Dehn filling along the longitude of a solid torus we obtain $S^3$. Therefore, this is the same as saying that if we reverse the orientation and crossings of $L$ and then perform $0$--framed Dehn surgery on $S^3$ along both components, we get back $S^3$. Further, in this new $S^3$, we can find the components of the link corresponding to $\overline{P}$. In the proof above these were the images of the curves $\ell(V)$ and $m(P)$. In the framework of links, these are the images of the meridians of the two components of $L$ (see Figure~\ref{fig:schematicinversepicture}) -- the meridian of the first component of $L$ is the second component of $\overline{L}$ and the meridian of the second component of $L$ is the first component of $\overline{L}$. Therefore, we have proved the following proposition. 

\begin{figure}[t]
\includegraphics{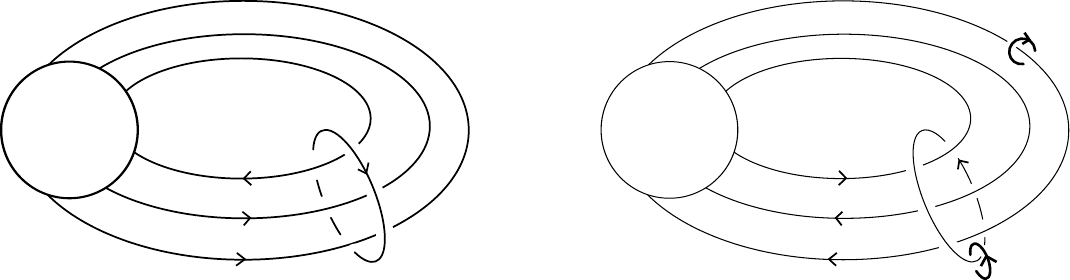}
\put(-4.1,0.55){$\mathscr{P}$}
\put(-2.7,1){$\widetilde{P}$}
\put(-2.75,0){$\eta$}
\put(-3.4,-0.4){(a)}
\put(-1.75,0.55){$m\mathscr{P}$}
\put(-0.7,0.5){0}
\put(-1.25,0){0}
\put(-0.3, -0.12){$\widetilde{\overline{P}}$}
\put(-0.125,0.95){$\overline{\eta}$}
\put(-1,-0.4){(b)}
\caption{(a) The 2--component link $(\widetilde{P},\eta)$ corresponds to the pattern $P$ (see Figure~\ref{fig:patterntolink}). Recall that the circle containing $\mathscr{P}$ denotes a tangle. (b) The circle containing $m\mathscr{P}$ indicates the tangle obtained from $\mathscr{P}$ in the previous panel by reversing all the crossings. The curves decorated with $0$'s give a surgery diagram for $S^3$. The 2--component link $(\widetilde{\overline{P}},\overline{\eta})$ (drawn in heavier weight) corresponds to the pattern $\overline{P}$.} \label{fig:schematicinversepicture}
\end{figure}
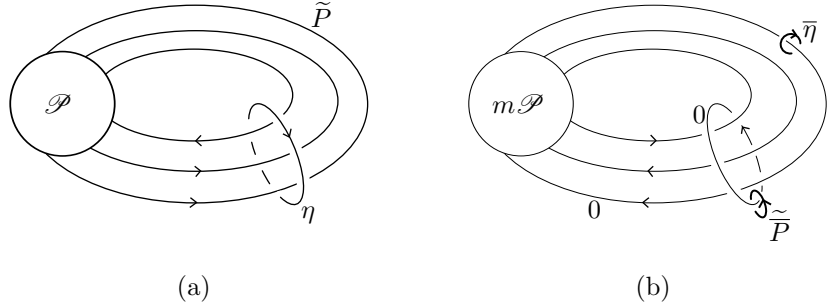

\begin{proposition}{\label{prop:inverseaslink}}
Let $P\subseteq V=S^1\times D^2$ be in $\S_\Z$. Assume that $m(P)$ is in the normal subgroup of $\pi_1(E(P))$ generated by $m(V)$ and $(\widetilde{P},\eta)$ is the 2--component link corresponding to $P$ where $\eta$ is unknotted in $S^3$. Then the inverse of $P$ is given by the link $\left(\widetilde{\overline{P}}, \overline\eta\right)$ in the surgery diagram for $S^3$ given by zero surgery on both components of the reverse mirror image of $(\widetilde{P},\eta)$, where $\widetilde{\overline{P}}$ is the meridian of $\eta$ and $\overline{\eta}$ is the meridian of $\widetilde{P}$. \end{proposition}

\begin{figure}[b]
        \centering
        \begin{subfigure}[t]{0.3\textwidth}
        \centering
		\includegraphics[width=1.5in]{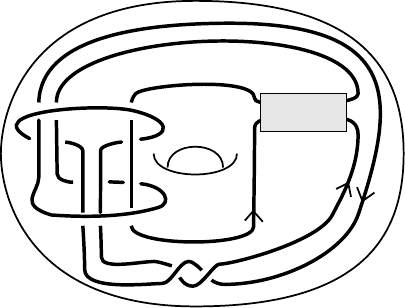}
		\put(-0.535,0.7){\tiny $2m+1$}
               	 \caption*{The satellite operator $P_m$}
        \end{subfigure}%
        \hspace{10pt}
        \begin{subfigure}[t]{0.3\textwidth}
        	      \centering
	      \includegraphics[width=1.5in]{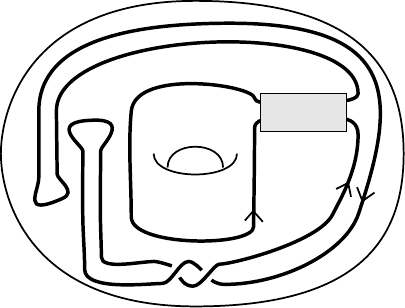}
	      \put(-0.535,0.7){\tiny $2m+1$}
               \caption*{The result of sliding $P_m$ over the meridian of $V(m)$}
        \end{subfigure}%
        \hspace{10pt}
        \begin{subfigure}[t]{0.3\textwidth}
        	      \centering
	     \includegraphics[width=1.5in]{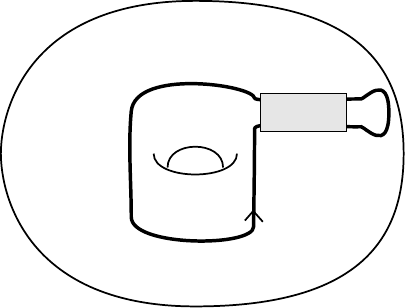}
	      \put(-0.535,0.7){\tiny $2m+1$}
               \caption*{The result of a further isotopy}
        \end{subfigure}%
        \caption{The patterns $P_m$ satisfy the requirements of Theorem~\ref{thm:honestinverse}.}\label{fig:proofofsurjex}
\end{figure}

\begin{remark}\label{rem:surjexinverse}Let $P_m\subseteq V(P_m)=S^1\times D^2$ be the patterns shown in Figure~\ref{fig:surjex} (and again in Figure~\ref{fig:proofofsurjex}). Note that each $P_m$ satisfies the requirements of Theorem~\ref{thm:honestinverse}, as follows. It suffices to show that $m(P_m)$ is nullhomotopic in the 3--manifold $N$ obtained from $E(P_m)$ by adding a 2--handle along $m(V(P_m))$. The result of sliding $P_m$ over this 2--handle twice (isotopies in $N$) is depicted in  Figure~\ref{fig:proofofsurjex}.  In the result of the isotopy, the meridian of $P_m$ cobounds an annulus with the meridian of $V(P_m)$ and so bounds a disk in $N$. As a result, we can use Proposition~\ref{prop:inverseaslink} to construct inverses for the patterns $\{P_m\}_{m\geq 0}$ shown in Figure~\ref{fig:surjex}. This is indicated in Figure~\ref{fig:exampleinverse}. \end{remark}

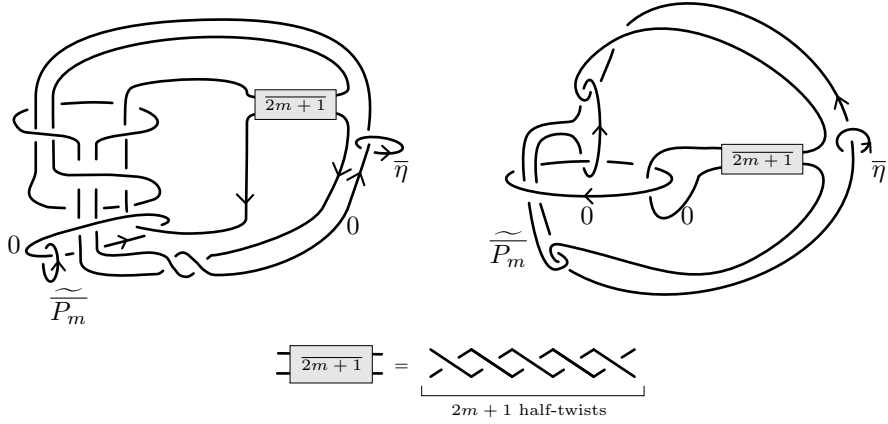
\begin{figure}
\includegraphics[width=4.5in]{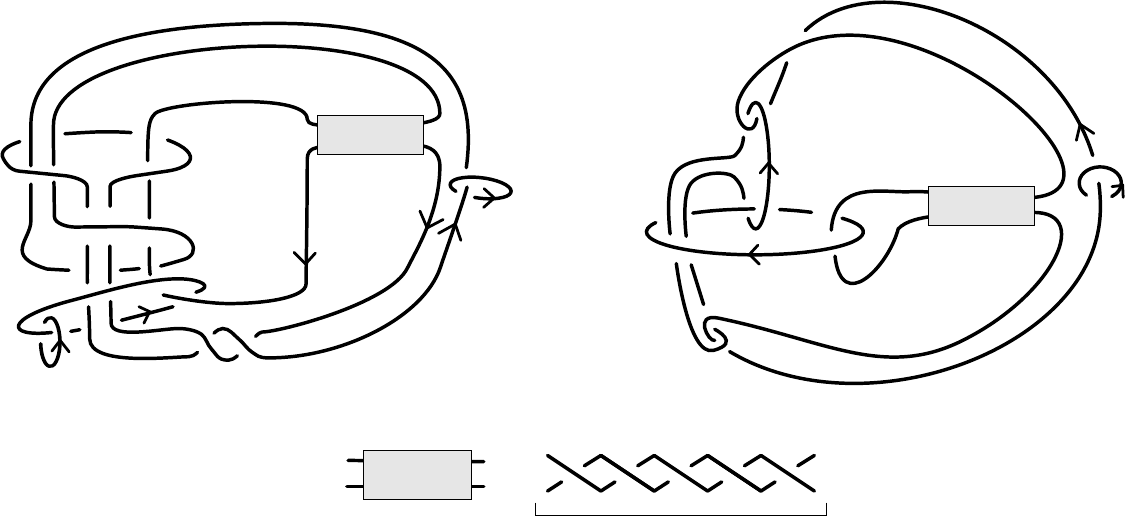}
\put(-3.175,1.5){\tiny $\overline{2m+1}$}
\put(-2.985,0.135){\tiny $\overline{2m+1}$}
\put(-0.73,1.21){\tiny $\overline{2m+1}$}
\put(-2.5,0.135){\tiny $=$}
\put(-2.2,-0.1){\tiny $2m+1$ half-twists}
\put(-2.75,0.85){0}
\put(-4.525,0.75){0}
\put(-1.525,0.9){0}
\put(-1,0.9){0}
\put(-4.3,.4){$\widetilde{\overline{P_m}}$}
\put(-2.5,1.15){$\overline{\eta}$}
\put(-2,.7){$\widetilde{\overline{P_m}}$}
\put(0,1.15){$\overline{\eta}$}
\caption{Left: By Proposition~\ref{prop:inverseaslink}, the link $(\widetilde{\overline{P_m}},\overline{\eta})$ in this surgery diagram represents the inverse of the pattern $P_m$, for $m\geq 0$. Right: This diagram is obtained from the one in the previous panel by handle-slides and isotopy, and we see a standard picture of $S^3$ (notice that the curves marked with $0$'s form a Hopf link). To get a picture of the inverse pattern as a link, we simply need to slide the undecorated curves away from the surgery curves. This readily yields a picture of a link in $S^3$.}\label{fig:exampleinverse}
\end{figure}

Theorem~\ref{thm:honestinverse} also gives a sufficient condition for satellite operators to be bijective, as follows. 

\begin{corollary}\label{cor:givessurj} Let $P\subseteq V=S^1\times D^2$ be in $\S_\Z$. If $m(P)$ is in the normal subgroup of $\pi_1(E(P))$ generated by $m(V)$ then $P:\C_*\to \C_*$ is bijectve for $*\in\{\ex,\top, \Z\}$.\end{corollary}

\begin{proof}Surjectivity follows from Proposition~\ref{prop:surjectivecharacterization} and Theorem~\ref{thm:honestinverse} since if $E(P)^{-1}=E(\overline{P})$ for some $\overline{P}\in\S_\Z$, then $E(P)^{-1}(\Psi(\C_*)) = (E(\overline P))(\Psi(\C_*))\subseteq \Psi(\C_*)$. Any $P\in\S_\Z$ is injective on $\C_\Z$. By Theorem~\ref{thm:honestinverse}, $P$ is strong winding number $\pm 1$ and therefore, is injective on $\C_\ex$ and $\C_\top$. \end{proof}

Before we provide the promised examples of bijective operators, we will need the following lemma, which provides an extension of the operation $P\mapsto \tau(P)$ of ``twisting a pattern'' to the setting of generalized satellite operators. The function $\tau:\S\to\S$ gives a full right-handed twist to each pattern, as shown in Figure~\ref{fig:twist}. 

\begin{figure}[t]
        \centering
        \begin{subfigure}[b]{0.5\textwidth}
        \centering
		\includegraphics[width = 1.5in]{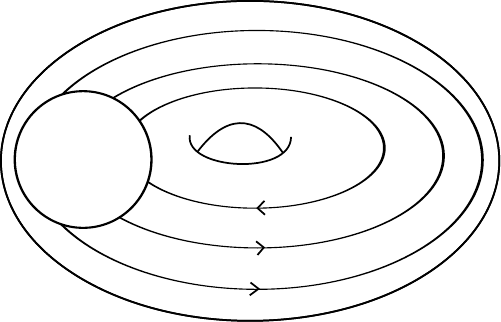}
		\put(-1.325,.45){$\mathscr{P}$}
               	 \caption*{A pattern $P$}
        \end{subfigure}%
        \begin{subfigure}[b]{0.5\textwidth}
        	      \centering
	      \includegraphics[width = 1.5in]{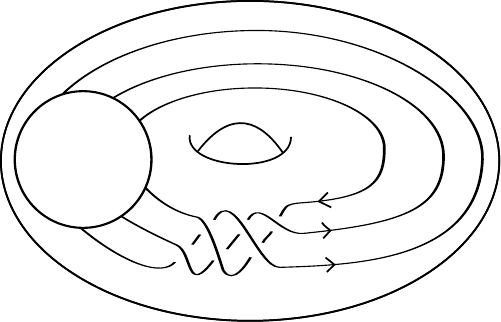}
	      \put(-1.325,.45){$\mathscr{P}$}
               \caption*{The twisted pattern $\tau(P)$}
        \end{subfigure}%
        \caption{Twisting a pattern.}\label{fig:twist}
\end{figure}

\begin{lemma}\label{lem:twist}
For $*\in \{\ex, \top, \Z\}$ there is a map $\widehat\tau:\Shat_* \to \Shat_*$ such that for all $P\in \S_\str$, $\widehat\tau(E(P)) = E(\tau(P))$ as elements of $\Shat_*$.
\end{lemma}

\begin{proof}
Let $P$ be a pattern in the solid torus $V$, and $\tau(P)$ be the corresponding twisted pattern.  Let $f:V\to V$ be the homeomorphism given by a negative meridional Dehn twist.  Notice that $f$ sends $\tau(P)$ to $P$.  Thus, $f$ restricts to a homeomorphism $E(\tau(P))\rightarrow E(P)$.  This homeomorphism sends $m(V)$ to $m(V)$ and $\ell(V)$ to $\ell(V)-m(V)$. Since $f_*$ is well-defined on homology classes, $f$ sends $m(\tau(P))$ to $m(P)$ and $\ell(\tau(P))$ to $\ell(P)-m(P)$.  Let $\phi:T\to T$ be the homeomorphism of the torus sending $m$ to $m$ and $\ell$ to $\ell-m$.  As homology cylinders, $P$ and $\tau(P)$ are given by $(E(P), i_+, i_-)$ and $(E(P), j_+, j_-)$ where $j_\epsilon = i_\epsilon\circ \phi$ for $\epsilon \in \{+,-\}$.

For any homology cylinder $(M,i_+,i_-)$, define $\widehat \tau(M, i_+, i_-) = (M,  i_+ \circ \phi,  i_-\circ \phi)$.  By the preceding paragraph, for any pattern $P$,  $\widehat{\tau}(E(P)) = E(\tau(P))$.  It remains only to show that $\widehat{\tau}$ is well-defined modulo $*$--cobordism.  Assume that $W$ is a $*$--cobordism between $(M, i_+, i_-)$ and $(N, j_+, j_-)$.  Taking advantage of the fact that $\phi:T\to T$ is a diffeomorphism, we see that
$$
\begin{array}{rcl}\bdry W &=& \bigslant{M\sqcup -N}{i_+(x)=j_+(x),i_-(x)=j_-(x), \forall x\in T}
\\&\cong& \bigslant{M\sqcup -N}{ i_+( \phi(x))= j_+( \phi(x)),  i_-( \phi(x))= j_-( \phi(x)), \forall x\in T}\end{array}
$$
Therefore, $W$ is also a $*$--cobordism between $\widehat\tau(M, i_+, i_-)$ and $\widehat\tau(N, j_+, j_-)$.  
\end{proof}

We are now ready to construct examples of satellite operators which yield bijective functions on knot concordance, and are distinct from connected sum operators. 

\begin{corollary}\label{cor:surjexample}Fix $m\geq 0$. The pattern $P_m\subseteq V(P_m)=S^1\times D^2$ shown in Figure~\ref{fig:surjex} (and again in Figure~\ref{fig:proofofsurjex}) gives a bijective map $P_m:\C_*\to \C_*$ for $*\in\{\ex,\top,\Z\}$; moreover as elements of $\Shat_\Z$, $E(P_m)\neq E(Q_J)$ for all knots $J$. \end{corollary}
\begin{proof}We already know that each $P_m$ satisfies the requirements of Theorem~\ref{thm:honestinverse} and Corollary~\ref{cor:givessurj} from Remark~\ref{rem:surjexinverse}. This gives the first statement. 

In order to see the second result notice that the twisting map $\tau$ of Figure~\ref{fig:twist} and Lemma~\ref{lem:twist} sends $Q_J$ to $Q_J$ for any knot $J$.  It suffices then to prove that $\tau(P_m)\neq P_m$ in $\Shat_\Z$.  In order to see this first observe that $P_m(U)$ is smoothly slice.   However, according to an Alexander polynomial computation  the knots $(\tau(P_m))(U)$ are not slice for $m\geq 0$ (see also ~\cite[Theorem 3.6]{AJOT13}).  Since the Alexander polynomial is an obstruction to being slice in a $\Z$--homology sphere, we conclude that as maps on $\C_\Z$, $P_m$ and $\tau(P_m)$ disagree, so that as elements of $\Shat_\Z$, $P_m\neq \tau(P_m)$.  This completes the proof.\end{proof}

In passing, we note that by Remark~\ref{rem:concthencob}, the above result implies that the link $(P_m,\eta(V(P_m)))$ is not (smoothly, exotically, topologically, or $\Z$--) concordant to the link $(Q_J,m(V(Q_J)))$ for any knot $J$ and $m\geq 0$ (see also~\cite[Proposition 2.3]{CDR14}). It is also worth noting that even though $P_m$ and $Q_J$ are distinct for all knots $J$ and $m\geq 0$, it is still possible that $P_m(K)=J\#K$ for some fixed knot $J$, and any knot $K$. 

We end this section with the following result, leading to a corollary for satellite operators with winding number other than $\pm 1$.

\begin{proposition}\label{prop:surjectiveandhitsunknot}There exists a satellite operator in $\S_\str$ for which the map on $\C_*$ is not surjective if and only if there exists a satellite operator in $\S_\str$ for which the unknot is not in the image of the map on $\C_*$, for $*\in\{\ex,\top\}$. 

Similarly, for $R$ a localization of $\Z$, there exists a satellite operator in $\S_R$ for which the map on $\C_R$ is not surjective if and only if there exists a satellite operator in $\S_R$ for which the unknot is not in the image of the map on $\C_R$.\end{proposition}

\begin{proof}For the $\Rightarrow$ direction, let $P$ be a satellite operator that is not surjective on $\C_*$ for some $*$, i.e.\ there exists a knot $J$ such that $P(K)$ is not concordant to $J$, in the appropriate sense dictated by the value of $*$, for any knot $K$. Then the satellite operator $Q_{-J}\star P$ does not have the unknot in the image of its induced map on $\C_*$, since if $(Q_{-J}\star P)(K)=-J\#P(K)$ were concordant to the unknot for some $K$, then $P(K)$ would be concordant to $J$. The $\Leftarrow$ direction is trivial. \end{proof}

\begin{corollary}For any integer $n$, with $\lvert n\rvert >1$ and $R=\Z[\frac{1}{n}]$, there exist  satellite operators in $\S_R$ which do not have the unknot in their image as a map on $\C_R$.\end{corollary}

\section{Concordance to knots in $S^3$ and surjectivity of satellite operators}\label{sec:kirbyconnection}

Akbulut conjectured that there exists a winding number one satellite operator $P$ which does not have the unknot in its image under $P:\C_\ex\to\C_\ex$.  By Proposition~\ref{prop:surjectiveandhitsunknot} this conjecture is equivalent to the conjecture that not all winding number one satellite operators are surjective. We restate Akbulut's conjecture in these terms.

\begin{conjecture}[Problem 1.45 of \cite{kirbylist}]\label{conj:Akbulutconj}There is a satellite operator of winding number one, $P$, such that $P:\C_*\to \C_*$ is not surjective, for $*\in\{\ex,\top\}$.\end{conjecture}

Consider a knot $K$ in a homology sphere $M$. Then $(K,M)$ gives a class in $\Chat_*$ for $*\in\{\ex,\top\}$ and one may ask whether there is some knot $K'\subseteq S^3$ such that $(K,M)$ and $(K',S^3)$ are equivalent in $\Chat_*$. In the PL category, this forms Problem~1.31 of~\cite{kirbylist}. We restate this as a conjecture. 

\begin{conjecture}[Problem 1.31 of \cite{kirbylist}]\label{conj:Matsumotoconj} The image of $\Psi:\C_*\to \Chat_*$ is the set of all concordance classes $(K,M)$ of knots $K$ in 3--manifolds $M$ where $M$ is $*$--cobordant to $S^3$.\end{conjecture}

We can use the group action given in the main theorem to prove the following relationship between the two conjectures above.

\begin{proposition}\label{prop:conjectures}For $P\in \S_\str$ and any $K\in \C_\ex$ (resp.\ $\C_\top$), if $K\notin \im(P:\C_\ex\to\C_\ex)$ (resp.\ $\im(P:\C_\top\to\C_\top))$, then the knot $E(P)^{-1}(\Psi(K))$ is not in the image of $\Psi:\C_\ex\to\Chat_\ex$ (resp.\ $\C_\top\to\Chat_\top$) and moreover, is contained in a 3--manifold smoothly (resp.\ topologically) homology cobordant to $S^3$. \end{proposition}

\begin{proof} 
To see the first claim, suppose that $E(P)^{-1}(\Psi(K))$ is equal in $\Chat_*$ to $\Psi(J)$ for some $J\in\C_*$ (for $*=\ex$ or $\top$) then $\Psi(K)= E(P)(\Psi(J))$. Then by the diagrams in \eqref{diag:maindiagram}, since $\Psi$ is injective, $K=P(J)$ and therefore, $K\in \im(P:\C_*\to\C_*)$. The second statement follows from the following lemma since $\Psi(K)=(K,S^3)$. \end{proof}

\begin{lemma}If $P\in \S_\str$ and $(K,Y)\in \Chat_\ex$ (resp.\ $\Chat_\top$), then $E(P)^{-1}(K,Y)$ is a knot in a 3--manifold which is smoothly (resp.\ topologically) homology cobordant to $Y$.
\end{lemma}

\begin{proof}
By definition. $E(P)(K,Y)$ is a knot in the 3--manifold
$$\overline{Y}= S^1\times D^2 \underset{\bdry D^2\sim m(P)}{\cup} E(P) \underset { \ell(V) \sim \ell(K)} {\underset {
m(V) \sim m(K)}{\cup} }Y-K.$$
But $S^1\times D^2 \,\cup\, E(P)$ is just a solid torus with meridian $m(V)$ and therefore, these gluing instructions cut a solid torus out of $Y$ and then glue it back in the same way. Therefore, $\overline{Y}$ is diffeomorphic to $Y$ and $E(P)(K,Y)$ is a knot in $Y$.

Let $E(P)^{-1}(K,Y) = (K',Y')$ and $E(P)(K',Y') = (K'',Y'')$.  By the preceding paragraph, $Y''$ is diffeomorphic to $Y'$.  Since $E(P)\circ E(P)^{-1}$ is the identity map on $\Chat_\ex$ (resp.\ $\Chat_\top$), $K$ is concordant to $K''$ in a smooth (resp.\ topological) homology cobordism between $Y$ and $Y''$, and hence $Y$ is  smoothly (resp.\ topologically) homologically cobordant to $Y''=Y'$. Since $E(P)^{-1}(K,Y) = (K',Y')$, the proof is completed. \end{proof}

The above proposition shows that if a strong winding number one satellite operator $P$ fails to be surjective on $\C_\ex$ (resp.\ $\C_\top$), i.e.\ there is some $K\in\C_\ex$ (resp.\ $\C_\top$) such that $K\neq P(J)$ for all knots $J$, then there exists a knot $K'$ in a 3--manifold $Y'$ smoothly (resp.\ topologically) concordant to $S^3$, such that $(K',Y')$ is not exotically (resp.\ topologically) concordant to any knot in $S^3$, where $(K',Y')=E(P)^{-1}(K,S^3)$.

\bibliographystyle{alpha}
\bibliography{bib}

\begin{thebibliography}{CFHH13}

\bibitem[AJOT13]{AJOT13}
Tetsuya Abe, In~Dae Jong, Yuka Omae, and Masanori Takeuchi.
\newblock Annulus twist and diffeomorphic 4-manifolds.
\newblock {\em Math. Proc. Cambridge Philos. Soc.}, 155(2):219--235, 2013.

\bibitem[AK79]{AkKir79}
Selman Akbulut and Robion Kirby.
\newblock Mazur manifolds.
\newblock {\em Michigan Math. J.}, 26(3):259--284, 1979.

\bibitem[Akb91]{Ak91}
Selman Akbulut.
\newblock A fake compact contractible {$4$}-manifold.
\newblock {\em J. Differential Geom.}, 33(2):335--356, 1991.

\bibitem[AY08]{AkYas08}
Selman Akbulut and Kouichi Yasui.
\newblock Corks, plugs and exotic structures.
\newblock {\em J. G\"okova Geom. Topol. GGT}, 2:40--82, 2008.

\bibitem[CDR14]{CDR14}
Tim~D. Cochran, Christopher~W. Davis, and Arunima Ray.
\newblock Injectivity of satellite operators in knot concordance.
\newblock {\em Journal of Topology}, 7(4):948--964, 2014.

\bibitem[CFHH13]{CFHeHo11}
Tim~D. Cochran, Bridget~D. Franklin, Matthew Hedden, and Peter~D. Horn.
\newblock Knot concordance and homology cobordism.
\newblock {\em Proc. Amer. Math. Soc.}, 141(6):2193--2208, 2013.

\bibitem[CHL11]{CHL11}
Tim~D. Cochran, Shelly Harvey, and Constance Leidy.
\newblock Primary decomposition and the fractal nature of knot concordance.
\newblock {\em Math. Ann.}, 351(2):443--508, 2011.

\bibitem[COT04]{CT04}
Tim~D. Cochran, Kent~E. Orr, and Peter Teichner.
\newblock Structure in the classical knot concordance group.
\newblock {\em Comment. Math. Helv.}, 79(1):105--123, 2004.

\bibitem[Fre82]{Free82}
Michael~Hartley Freedman.
\newblock The topology of four-dimensional manifolds.
\newblock {\em J. Differential Geom.}, 17(3):357--453, 1982.

\bibitem[Har08]{H08}
Shelly~L. Harvey.
\newblock Homology cobordism invariants and the {C}ochran-{O}rr-{T}eichner
  filtration of the link concordance group.
\newblock {\em Geom. Topol.}, 12(1):387--430, 2008.

\bibitem[Hem04]{Hemp04}
John Hempel.
\newblock {\em 3-manifolds}.
\newblock AMS Chelsea Publishing, Providence, RI, 2004.
\newblock Reprint of the 1976 original.

\bibitem[HK12]{HeK12}
Matthew Hedden and Paul Kirk.
\newblock Instantons, concordance, and {W}hitehead doubling.
\newblock {\em J. Differential Geom.}, 91(2):281--319, 2012.

\bibitem[Kir97]{kirbylist}
Rob Kirby, editor.
\newblock {\em Problems in low-dimensional topology}, volume~2 of {\em AMS/IP
  Stud. Adv. Math.}
\newblock Amer. Math. Soc., Providence, RI, 1997.

\bibitem[Lev01]{Lev01}
Jerome Levine.
\newblock {Homology cylinders: an enlargement of the mapping class group}.
\newblock {\em Algebraic and Geometric Topology}, 1:243--270, 2001.

\bibitem[Lev14]{Lev14}
Adam~Simon Levine.
\newblock Non-surjective satellite operators and piecewise-linear concordance.
\newblock Preprint, available at http://arxiv.org/abs/1405.1125, 2014.

\bibitem[Lit79]{Li77}
Richard~A. Litherland.
\newblock Signatures of iterated torus knots.
\newblock In {\em Topology of low-dimensional manifolds ({P}roceedings {S}econd
  {S}ussex {C}onference, {C}helwood {G}ate, 1977)}, volume 722 of {\em Lecture
  Notes in Mathematics}, pages 71--84. Springer, Berlin, 1979.

\bibitem[LM85]{LivM85}
Charles Livingston and Paul Melvin.
\newblock Abelian invariants of satellite knots.
\newblock In {\em Geometry and Topology}, volume 1167 of {\em Lecture Notes in
  Mathematics}, chapter~13, pages 217--227. Springer Berlin / Heidelberg, 1985.

\bibitem[Rol90]{Ro90}
Dale Rolfsen.
\newblock {\em Knots and links}, volume~7 of {\em Mathematics Lecture Series}.
\newblock Publish or Perish Inc., Houston, TX, 1990.
\newblock Corrected reprint of the 1976 original.

\bibitem[Sch54]{Schu54}
Horst Schubert.
\newblock \"{U}ber eine numerische {K}noteninvariante.
\newblock {\em Math. Z.}, 61:245--288, 1954.

\end{thebibliography}
\noindent
\end{document}